\documentclass[11pt]{amsart}

\usepackage{amsmath,amssymb,amsthm,mathtools,yhmath}
\usepackage{enumitem}
\usepackage{microtype}
\usepackage[colorlinks=true,linkcolor=blue,citecolor=blue,urlcolor=blue]{hyperref}
\usepackage[margin=1.15in]{geometry}

\numberwithin{equation}{section}
\newtheorem{theorem}{Theorem}[section]
\newtheorem{lemma}[theorem]{Lemma}

\newtheorem{proposition}[theorem]{Proposition}
\theoremstyle{definition} 

\newtheoremstyle{remarkstyle} 
   {}   
   {}   
  {\normalfont}  
  {}      
  {\itshape}  
  {.}     
  { }     
  {}      
\theoremstyle{remarkstyle}
\newtheorem*{remark}{Remark} 

\allowdisplaybreaks[4]

\newcommand{\R}{\mathbb{R}}
\newcommand{\C}{\mathbb{C}}

\newcommand{\Hh}{\mathbb{H}}
\newcommand{\diam}{\operatorname{diam}}

\newcommand{\id}{\operatorname{id}}
\newcommand{\Mod}{\operatorname{Mod}}
\newcommand{\esssup}{\operatorname*{ess\,sup}}

\title[Asymptotic symmetry of quasilines]
{Weak Asymptotic Symmetry of Quasisymmetric Embeddings on Quasilines}

\author{Katsuhiko Matsuzaki}
\address{Department of Mathematics, School of Education, Waseda University, Tokyo 169-8050, Japan}
\email{matsuzak@waseda.jp}
\author{Fei Tao}
\address{Beijing International Center for Mathematical Research, Peking University, Beijing 100871, P. R. China}
\email{ferrytau@pku.edu.cn}
\date{}

\thanks{The first author is partially supported by Japan Society for the Promotion of Science (KAKENHI 23K25775 and 23K17656); the second author is partially supported by National Key R \& D Program of China (2025YFA1017500)}


\begin{document}

\begin{abstract}
We investigate the relationship between weak asymptotic symmetry ($\mathrm{WAS}$) and asymptotic symmetry ($\mathrm{AS}$) for quasisymmetric maps associated with planar quasilines. To this end, we introduce a formally weaker condition, called equidistant weak asymptotic symmetry ($\mathrm{EWAS}$), and prove that, for quasisymmetric embeddings of $\mathbb{R}$ into $\mathbb{C}$, the three conditions $\mathrm{AS}$, $\mathrm{WAS}$, and $\mathrm{EWAS}$ are equivalent.

We then establish that $\mathrm{WAS}$ implies $\mathrm{AS}$ for quasisymmetric homeomorphisms from an arbitrary quasiline onto $\mathbb{R}$. More generally, if $h\colon\Gamma_1\to\Gamma_2$ is a quasisymmetric homeomorphism between quasilines and $\Gamma_1$ is asymptotically conformal, then $\mathrm{WAS}$ implies $\mathrm{AS}$. A formally dual statement holds when $\Gamma_2$ is asymptotically conformal, provided that $h$ is uniformly continuous. In the compact setting of bounded quasicircles, these results yield an answer to the $\mathrm{WAS}$--$\mathrm{AS}$ problem posed by Brania and Yang.
\end{abstract}

\maketitle

\section{Introduction}
Let $X,Y$ be metric spaces in general with distance $d(x_1,x_2)$ denoted by $|x_1-x_2|$. 
An embedding $h\colon X\to Y$ is called \emph{quasisymmetric} if there is a homeomorphism $\eta\colon[0,\infty)\to[0,\infty)$, called a \emph{distortion function}, such that
\[
 \frac{|a-x|}{|b-x|}\leq t\quad\implies\quad\frac{|h(a)-h(x)|}{|h(b)-h(x)|}\le\eta(t)
\]
for any three distinct points $a,b,x\in X$. 
The concept of quasisymmetry was first introduced in a weaker form
by Beurling and Ahlfors \cite{BA} for self-mappings of the real line $\mathbb{R}$,
and was later extended to general metric spaces by Tukia and V\"ais\"al\"a \cite{TukiaVaisala}. 
For self-mappings of $\mathbb C$, the equivalence between quasiconformality and quasisymmetry follows from the modulus characterization of quasiconformality; see, for example, \cite{LV71, AIM}. 

A Jordan curve in $\mathbb C$ is said to be a \emph{quasicircle} if 
it admits a quasisymmetric parametrization by the unit circle $\mathbb S$.
A \emph{quasiline} means a generalized quasicircle through $\infty$ in the Riemann sphere that arises as the image of a quasisymmetric embedding $f\colon\mathbb R\to \mathbb{C}$. 
This always extends to a quasiconformal self-homeomorphism of $\mathbb C$ by Tukia \cite{TukiaExtension}.
See also \cite[Proposition~5.10 and Theorem~5.11]{Pom}.

For any two points $a,b$ on a quasiline $\Gamma \subset \mathbb C$, let $\Gamma[a,b]$ denote the bounded subarc joining them. 
Quasilines can also be characterized geometrically by the \emph{bounded-turning} condition \cite{Ahlfors}: there is $C\ge1$ such that
\[
 \diam\Gamma[a,b]\leq C|a-b|
\]
for any $a,b\in\Gamma$. 
The constant $C \geq 1$ can be estimated by the maximal dilatation $K \geq 1$ of a possible quasiconformal extension of $f$ to $\C$, and vice versa.

The boundary behavior of quasiconformal mappings leads to an asymptotic version of conformality.  
A quasiconformal mapping is \emph{asymptotically conformal} if its complex dilatation tends to zero as approaching to the boundary.   
The corresponding geometric formulations for their images of the boundary were investigated in \cite{Car67,GS,P}. 
A quasiline $\Gamma$ is called \emph{asymptotically conformal} if
\begin{equation} \label{def:symmetric-quasiline}
\omega_{\Gamma}(r)\coloneqq\sup_{a,b\in\Gamma,\ 0<|a-b|\leq r}\max_{w\in\Gamma[a,b]}\left(\frac{|a-w|+|w-b|}{|a-b|}-1\right)\to0\qquad(r\to0).
\end{equation}
This serves as the unbounded analogue of the asymptotically conformal condition for quasicircles introduced by Pommerenke \cite{P}, and is also referred to as the symmetric quasicircle condition by Brania and Yang \cite{BraniaYang}.

To characterize asymptotic conformality, the notion of asymptotically symmetric embeddings
was introduced by
\cite[Definition~2.1]{BraniaYang}. An embedding $h\colon X\to Y$ is called \emph{asymptotically symmetric}, abbreviated as $\mathrm{AS}$, if for every $\varepsilon>0$ and every $t>0$ there exists $\delta>0$ such that
\begin{equation}\label{eq:AS}
 \frac{|a-x|}{|b-x|}\leq t\quad\implies\quad\frac{|h(a)-h(x)|}{|h(b)-h(x)|}\le(1+\varepsilon)t
\end{equation}
whenever $a,b,x\in X$ are distinct, and
$\diam\{a,b,x\}\leq\delta$. 
Condition \eqref{eq:AS} is an infinitesimal form of quasisymmetry in which the distortion function approaches the identity uniformly at small scales. 
If \eqref{eq:AS} is required only for $t=1$, the map $h$ is called \emph{weakly asymptotically symmetric}, abbreviated as $\mathrm{WAS}$.

Brania and Yang \cite[Section~4]{BraniaYang} proved that $\mathrm{WAS}$ and $\mathrm{AS}$ are equivalent for self-homeomor\-phisms of $\mathbb S$.
Their proof relies on the argument involving the conformal welding of $\mathrm{WAS}$ and asymptotically conformal extensions of the conformal mappings given by Gardiner and Sullivan \cite[Theorem 6.3]{GS}.
Then, they posed the question of whether this equivalence extends to embeddings of Jordan curves in $\mathbb C$, or under what geometric conditions such an implication holds. 

We also consider a condition weaker than $\mathrm{WAS}$. In fact, the $\mathrm{WAS}$ condition in \cite{Car67, GS} for quasisymmetric self-homeomorphisms 
of $\mathbb S$ and $\mathbb R$ was in this form.
In the equidistant case, the condition $\mathrm{WAS}$ provides a two-sided control of the distortion by interchanging $a$ and $b$. 
This leads to the following quantitative formulation:
For $r>0$, the \emph{symmetric distortion} is defined by
\begin{equation*}
 H_h^\mathrm{eq}(r)=\sup\left\{\max\left(\frac{|h(a)-h(x)|}{|h(b)-h(x)|},\frac{|h(b)-h(x)|}{|h(a)-h(x)|}\right):
 \begin{array}{c}
 a,b,x\in X,\ |a-x|=|b-x|,\\
 0<\diam\{a,b,x\}\leq r
 \end{array}
 \right\}.
\end{equation*}
We say that $h$ is \emph{equidistant weakly asymptotically symmetric}, abbreviated as $\mathrm{EWAS}$, if 
\[H_h^\mathrm{eq}(r) \to  1\quad \text{as}\quad r \to 0.\]
In particular, for an embedding $f\colon\R\to\C$,  the $\mathrm{EWAS}$ condition simplifies to
\begin{equation*}
 \lim_{r\to 0}\sup_{x\in\R,\, 0<s\leq r}\frac{|f(x+s)-f(x)|}{|f(x)-f(x-s)|}=1.
\end{equation*}

Directly from the definitions, we have the implications
\[
 \mathrm{AS}\implies\mathrm{WAS}\implies\mathrm{EWAS}.
\]
We note that, on a noncompact space, $\mathrm{AS}$ need not imply global quasisymmetry, which is therefore retained as a separate hypothesis.

Our first main result demonstrates that for quasisymmetric embeddings of $\mathbb{R}$, this weaker condition is already sufficient to guarantee asymptotic symmetry.

\begin{theorem}\label{thm:main-intro}
Let $f\colon\R\to\C$ be a quasisymmetric embedding.  
The following conditions are equivalent.
\begin{enumerate}[label=\textup{(\roman*)}]
\item $f$ is asymptotically symmetric $\mathrm{(AS)}$.
\item $f$ is weakly asymptotically symmetric $\mathrm{(WAS)}$.
\item The distortion of $f$ on symmetric triples tends uniformly to $1$ as the scale tends to $0$ $\mathrm{(EWAS)}$.
\end{enumerate}
\end{theorem}

Theorem~\ref{thm:main-intro} can be established using a quasiconformal extension theorem of Tukia \cite{TukiaExtension} based on the Beurling--Ahlfors formula.
Specifically, in Theorem~\ref{thm:Esharp-AC}, we show the following:
Let $f\colon\R\to\C$ be a quasisymmetric embedding satisfying condition \textup{(iii)} of Theorem~\ref{thm:main-intro}.  
The Tukia extension $\mathcal Ef\colon\C\to\C$ defined in Section~\ref{sec:tukia_extension_and_proof_of_theorem_ref_thm_main_intro} is quasiconformal and its complex dilatation $\mu_{\mathcal Ef}$ satisfies the asymptotically conformal condition
\[
 \esssup_{0<|\operatorname{Im}z|<r}
 |\mu_{\mathcal Ef}(z)| \to 0 \quad \text{as}\quad r \to 0.
\]

We adapt two complementary approaches from the real-line argument.  
One combines the asymptotically conformal extension as in Theorem \ref{thm:Esharp-AC} with the compactness of normalized quasiconformal maps in the proof of Proposition \ref{prop:AC-boundary-AS} to show that the boundary extension of an asymptotically conformal map is $\mathrm{AS}$.  
We also give in Section~\ref{sec:a_proof_by_modulus_of_curve_families} a modulus proof of Proposition \ref{prop:AC-boundary-AS} modeled on \cite[Theorem~3.1]{BraniaYang}: the relevant curve family is divided into curves staying in an almost conformal neighborhood and curves escaping that neighborhood, whose modulus is small by an annulus estimate.

The other works directly on the boundary without using the quasiconformal extension and the proof is shorter. 
This is based on the compactness argument of normalized mappings for symmetric triples of small scale. 
Then, the limiting map preserves all symmetric triples, and the rigidity theorem of McKemie and Vaaler \cite{McKemieVaaler} forces the limit to be affine. 
In fact, this method is used for the proofs of the other theorems in this paper.

The next result shows that the same conclusion holds for quasisymmetric homeomorphisms from a quasiline to $\mathbb R$.
In this case, we apply the same rigidity theorem to the limit of the normalized mappings applied to the inverse maps. The key observation is an intermediate-value argument along the inverse images of intervals.

\begin{theorem}\label{thm:onto-line-intro}
Let $\Gamma$ be a quasiline, and let $h\colon\Gamma\to\R$ be a quasisymmetric homeomorphism.  
If $h$ is $\mathrm{EWAS}$, then $h$ is $\mathrm{AS}$. Consequently, the conditions $\mathrm{EWAS}$, $\mathrm{WAS}$, and $\mathrm{AS}$ are equivalent.
\end{theorem}

In a more general setting, the same conclusion holds for a quasisymmetric homeomorphism $h\colon\Gamma_1 \to \Gamma_2$ between quasilines under the following circumstances. 
When $\Gamma_1$ is asymptotically conformal, every infinitesimal normalized limit of $\Gamma_1$ is the real line $\R$; the proof therefore reduces to the argument for Theorem~\ref{thm:main-intro}.

\begin{theorem}\label{thm:symmetric-source-intro}
Let $\Gamma_1$ and $\Gamma_2$ be quasilines, and assume that $\Gamma_1$ is an asymptotically conformal quasiline.  
Let $h\colon\Gamma_1\to\Gamma_2$ be a quasisymmetric homeomorphism. 
If $h$ is $\mathrm{EWAS}$, then $h$ is $\mathrm{AS}$.
Consequently, the conditions $\mathrm{EWAS}$, $\mathrm{WAS}$, and $\mathrm{AS}$ are equivalent.
\end{theorem}

Conversely, the case of maps onto $\mathbb{R}$ is handled via the inverse map, while asymptotic conformality of $\Gamma_2$ ensures that the normalized quasilines converge to $\mathbb{R}$.
However, an additional noncompactness issue arises here: a sequence of triples with diameters tending to zero need not have image diameters tending to zero. 
Thus, we assume uniform continuity of $h$, and the proof reduces to that of Theorem~\ref{thm:onto-line-intro}. 

\begin{theorem}\label{thm:symmetric-target-intro}
Let $\Gamma_1$ and $\Gamma_2$ be quasilines, and assume that $\Gamma_2$ is an
asymptotically conformal quasiline.  
Let $h\colon\Gamma_1\to\Gamma_2$ be a uniformly continuous quasisymmetric homeomorphism.  
If $h$ is $\mathrm{EWAS}$, then $h$ is $\mathrm{AS}$.
Consequently, the conditions $\mathrm{EWAS}$, $\mathrm{WAS}$, and $\mathrm{AS}$ are equivalent.
\end{theorem}

We mainly investigate quasisymmetric homeomorphisms between noncompact quasilines, but the compact cases for circles and bounded quasicircles follow from these noncompact cases either by lifting the homeomorphisms to the universal covers as equivariant mappings and applying the results for noncompact cases, or modifying the settings to compact cases and applying the arguments verbatim.
Thus, the results of Theorems \ref{thm:symmetric-source-intro} and \ref{thm:symmetric-target-intro} can be converted to the compact case as follows:
For bounded quasicircles $\Gamma_1$ and $\Gamma_2$ in $\C$ either of which is asymptotically conformal, any weakly asymptotically symmetric ($\mathrm{WAS}$) homeomorphism
$h\colon\Gamma_1\longrightarrow\Gamma_2$ is asymptotically symmetric ($\mathrm{AS}$).
This is given as Theorem \ref{thm:compact-one-symmetric}, which
is a partial answer to the aforementioned $\mathrm{WAS}$--$\mathrm{AS}$ problem of Brania and Yang.

The paper is organized as follows. 
Section~\ref{sec:preliminary_results} collects the preliminary results used throughout the paper, in particular,
the rigidity theorem of McKemie and Vaaler.
Section~\ref{sec:tukia_extension_and_proof_of_theorem_ref_thm_main_intro} develops the Tukia extension and proves Theorem~\ref{thm:main-intro}. 
Section~\ref{sec:proof_of_theorems_ref_thm_onto_line_intro} proves Theorems~\ref{thm:onto-line-intro},~\ref{thm:symmetric-source-intro}, and~\ref{thm:symmetric-target-intro}, and clarifies the noncompact and compact settings, thereby providing an answer to the $\mathrm{WAS}$--$\mathrm{AS}$ problem.
Finally, Section~\ref{sec:a_proof_by_modulus_of_curve_families} offers a modulus-based proof following the approach of Brania and Yang.

\section{Preliminary results}\label{sec:preliminary_results}

The rigidity theorem of McKemie and Vaaler will be used in the proofs of our main results, which gives the following classification
\cite[Theorem~1 and Corollary~2]{McKemieVaaler}.

\begin{theorem}[McKemie--Vaaler]\label{thm:MV}
Let $g\colon\R\to\C$ be continuous and satisfy
\begin{equation}\label{eq:isosceles}
 |g(x+s)-g(x)|=|g(x-s)-g(x)|\qquad(x,s\in\R).
\end{equation}
Then $g$ is constant, affine, or circular.  
In particular, every quasisymmetric embedding $g\colon\R\to\C$ satisfying \eqref{eq:isosceles} is affine.
\end{theorem}

More explicitly, the nonconstant alternatives in their classification are $g(x)=Ax+B$ and $g(x)=Ae^{2\pi i\theta x}+B$.  
An embedding of $\R$ into $\C$ that extends to $\infty$ is unbounded, so the circular alternative is excluded.
We say that a quasisymmetric embedding $g\colon\R\to\C$ is {\it weakly $1$-quasisymmetric} if it satisfies the condition \eqref{eq:isosceles}.

Theorem~\ref{thm:MV} is genuinely planar.  
Weakly $1$-quasisymmetric embeddings of $\R$ into higher-dimensional Euclidean spaces need not be affine; thus the planar rigidity is essential here.

\begin{remark}
For a continuous map $\gamma\colon \mathbb S \to \mathbb C$ satisfying
\begin{equation}\label{eq:isosceles-circle}
|\gamma(e^{2\pi i(x+s)})-\gamma(e^{2\pi ix})|=|\gamma(e^{2\pi i(x-s)})-\gamma(e^{2\pi ix})|\qquad(x,s\in\R),
\end{equation}
we can assert a similar conclusion to Theorem~\ref{thm:MV}.
This is because $g=\gamma \circ p$ for $p(x)=e^{2\pi ix}$ satisfies \eqref{eq:isosceles}.
In particular, every weakly $1$-quasisymmetric embedding of $\mathbb S$ is circular.
\end{remark}

We shall use the compactness argument for normalized quasisymmetric and quasiconformal mappings.
For those mappings $g$ on $\mathbb R$ and $\mathbb C$, the normalization is given by fixing three points $0$, $1$, and $\infty$.
The last condition is assumed as $\lim_{z \to \infty} g(z)=\infty$ and this is automatic for
quasisymmetric embeddings and quasiconformal self-homeomorphisms.
The convergence of quasisymmetric and quasiconformal mappings is uniform with respect to the spherical metric of 
the Riemann sphere $\widehat{\mathbb C}=\mathbb C \cup \{\infty\}$.
This implies local uniform convergence on each compact subset of $\mathbb C$ in the Euclidean metric.

The following facts follow from the usual normal-family theory; see, for example,
\cite{Ahlfors,VaisalaQC}.

\begin{lemma}\label{lem:normal-boundary}
Let $\eta$ be a distortion function. 
A family of $\eta$-quasisymmetric embeddings $g_n\colon\R\to\C$ normalized by $g_n(0)=0$, $g_n(1)=1$, and
$g_n(\infty)=\infty$ is compact under uniform convergence in the spherical metric. 
Furthermore, any limit of a convergent subsequence is also a normalized $\eta$-quasisymmetric embedding.
\end{lemma}

\begin{lemma}\label{lem:normal-qc}
Let $K\ge1$. 
A family of $K$-quasiconformal homeomorphisms $G_n\colon\C\to\C$ fixing $0,1,\infty$ is compact under uniform convergence in the spherical metric. 
In particular, every sequence $G_n$ in this family admits a subsequence, still denoted by $G_n$, such that both $G_n$ and $G_n^{-1}$ converge uniformly in the spherical metric.
If, in addition, the complex dilatations of $G_n$ satisfy
\[
 \esssup_{z\in E}|\mu_{G_n}(z)|\to 0 \quad(n \to \infty)
\]
for every compact subset $E\subset\C$, then every subsequential limit is conformal and hence, by the normalization, is the identity.
\end{lemma}

We say that a quasiline $\Gamma$ is \emph{normalized} if it contains $0$ and $1$ (and extends to $\infty$). 
A sequence of quasilines $\Gamma_n$ is said to converge to a quasiline $\Gamma_\infty$ if there exist quasisymmetric embeddings $g_n$ satisfying $g_n(\mathbb R) = \Gamma_n$ that converge uniformly in the spherical metric to a quasisymmetric embedding $g_\infty$ with $g_\infty(\mathbb R) = \Gamma_\infty$.

\begin{lemma}\label{lem:symmetric-blowup}
Let $\Gamma$ be an asymptotically conformal quasiline, and let $x_n, y_n \in \Gamma$ be sequences such that $r_n = |y_n-x_n| \to 0$.
Set
\[
 S_n(z)=\frac{z-x_n}{y_n-x_n}, \qquad \Gamma_n=S_n(\Gamma).
\]
Then every subsequential limit of $\Gamma_n$ in the spherical topology is $\R$.
\end{lemma}

\begin{proof}
Since $\Gamma$ is a quasiline and $S_n$ is a similarity, the curves $\Gamma_n$ are normalized quasilines with uniform quasisymmetric control, so a subsequence converges to a normalized quasiline $\Gamma_\infty$ by Lemma~\ref{lem:normal-boundary}.
Let $a,b$ be distinct points of $\Gamma_\infty$, and let $w$ lie on the compact subarc $\Gamma_\infty[a,b]$.  
Using the normalized quasisymmetric homeomorphisms $g_n\colon\R\to \Gamma_n$, choose points $a_n,b_n,w_n\in\Gamma_n$ converging respectively to $a,b,w$ and respecting the order on the compact subarcs.  

Let $\tilde{a}_n, \tilde{b}_n, \tilde{w}_n \in \Gamma$ denote their respective preimages under $S_n$. 
The distance between the preimages is given by 
\[
 |\tilde{a}_n-\tilde{b}_n| = r_n|a_n-b_n|.
\]
Since the sequence $|a_n-b_n|$ converges to $|a-b|$ and is thus bounded, and since $r_n \to 0$, we have $\rho_n \coloneqq  |\tilde{a}_n-\tilde{b}_n| \to 0$ as $n \to \infty$. 

The definition of $\omega_\Gamma$ in ~\eqref{def:symmetric-quasiline} therefore gives
\[
 \frac{|a_n-w_n|+|w_n-b_n|}{|a_n-b_n|} = \frac{|\tilde{a}_n-\tilde{w}_n|+|\tilde{w}_n-\tilde{b}_n|}{|\tilde{a}_n-\tilde{b}_n|} \leq 1 + \omega_\Gamma(\rho_n).
\]
This can be rewritten as
\[
 |a_n-w_n|+|w_n-b_n| \leq \bigl(1+\omega_\Gamma(\rho_n)\bigr)|a_n-b_n|.
\]
Since $\rho_n \to 0$ implies $\omega_\Gamma(\rho_n) \to 0$ by the asymptotic conformality of $\Gamma$,
passing to the limit as $n \to \infty$, we obtain
\[
 |a-w|+|w-b| \leq |a-b|.
\]
Combined with the triangle inequality, this yields
\[
 |a-w|+|w-b|=|a-b|.
\]
Hence $w$ belongs to the Euclidean segment $[a,b]$.  
Every finite compact subarc of $\Gamma_\infty$ is therefore a line segment. 
Thus $\Gamma_\infty$ is a straight line, and because it contains $0$ and $1$, it is $\R$.
\end{proof}

\begin{lemma}\label{lem:equidistant-approx}
Let $\Gamma_n$ be quasilines converging locally to $\R$, with the two components of $\Gamma_n\setminus\{c_n\}$ converging in order to the two half-lines determined by $x\in\R$, where $c_n\to x$.  
For every $s>0$ there are $a_n,b_n\in\Gamma_n$ such that
\[
 a_n\to x+s,\qquad b_n\to x-s,
 \qquad |a_n-c_n|=|b_n-c_n|.
\]
\end{lemma}

\begin{proof}
For every $s>0$, let $a_n$ and $b_n$ be the endpoints of the connected component of 
\[
\Gamma_n \cap \{z : |z-c_n| \leq s\}
\]
containing $c_n$, lying respectively on the positive and negative components of $\Gamma_n \setminus \{c_n\}$.
These points satisfy $|a_n-c_n|=|b_n-c_n|=s$.

To show that $a_n \to x+s$, observe first that the sequence $a_n$ is bounded since $|a_n - c_n| = s$ and $c_n \to x$. 
Suppose to the contrary that a subsequence $a_{n_k}$ stays away from $x+s$. By compactness, we can extract a further convergent sub-subsequence $a_{n_{k_j}}$ converging to some limit $a_\infty$. 

Since $\Gamma_n$ converges locally to $\mathbb{R}$ and $a_{n_{k_j}}$ lies on the positive component of $\Gamma_{n_{k_j}}\setminus\{c_{n_{k_j}}\}$, the limit point $a_\infty$ must lie on the positive component of $\mathbb{R}\setminus\{x\}$. 
Passing to the limit in $|a_{n_{k_j}} - c_{n_{k_j}}| = s$ yields $|a_\infty - x| = s$, which forces $a_\infty =x+s$. 
This contradicts the assumption that the subsequence stays away from $x+s$.

A completely analogous argument shows that $b_n \to x-s$. This completes the proof.
\end{proof}

\section{Tukia extension and proof of Theorem~\ref{thm:main-intro}}\label{sec:tukia_extension_and_proof_of_theorem_ref_thm_main_intro}

In this section, we prove Theorem~\ref{thm:main-intro} by using quasiconformal extension. 
The flow of the argument is that a quasisymmetric embedding $g\colon \mathbb R \to \mathbb C$ has an asymptoticaly conformal extension if $g$ is $\mathrm{EWAS}$, and if $g$ admits such a quasiconformal extension then it should be $\mathrm{AS}$.
We also provide an alternative proof without using quasiconformal extension.

\subsection{Tukia's quasiconformal extension}

We recall the construction in \cite[Sections~3--5]{TukiaExtension}.  Suppose first that a quasisymmetric embedding $g\colon\R\to\C$ is normalized by $g(0)=0, g(1)=1, g(\infty)=\infty$.
The curve $g(\mathbb R)$ is a quasiline.  
Let $\Omega_g^+$ denote the component of $\mathbb C \setminus g(\mathbb R)$ lying above $g(\mathbb R)$, and let
\[
 A_g\colon\Hh\to\Omega_g^+
\]
be the normalized conformal map fixing $0,1,\infty$ in the boundary sense.
Then
\[
 h_g=A_g^{-1}\circ g\colon\R\to\R
\]
is a normalized increasing quasisymmetric homeomorphism.

For an increasing homeomorphism $h\colon\R\to\R$, we define
\begin{equation*}
 B_h(x,y)=\frac12(\alpha_h+\beta_h)
 +i(\alpha_h-\beta_h),
\end{equation*}
where
\begin{equation}\label{eq:alphabeta}
 \alpha_h(x,y)=\int_0^1h(x+ty)\,dt,
 \qquad
 \beta_h(x,y)=\int_0^1h(x-ty)\,dt.
\end{equation}
This gives a modified Beurling--Ahlfors extension. The original definition used $\frac{i}{2}(\alpha_h-\beta_h)$ for the imaginary part, but we adjust it here so that $B_{\rm id}(x,y)=x+iy$. 
It is well known that the Beurling--Ahlfors extension is a quasiconformal homeomorphism of the plane, with maximal dilatation depending only on the equidistant weak quasisymmetry constant of $g$; see \cite[Section~6]{BA}. 
Our modified extension $B_h$ differs from the standard Beurling--Ahlfors extension only by post-composition with an affine map, and hence is quasiconformal as well.

The upper half-plane part of Tukia extension is
\begin{equation*}
 \mathcal E g|_{\Hh}=A_g\circ B_{h_g}.
\end{equation*}
Since $A_g$ is conformal and $B_{h_g}$ is quasiconformal, their composition is quasiconformal.
The lower half-plane part is defined analogously. 

A key property of Tukia extension is its affine covariance; see \cite[Section~5]{TukiaExtension}. More precisely, if $\phi\colon\R\to\R$ is an increasing affine translation and $\psi\colon\C\to\C$ is a similarity (conformal affine map), then, after normalization,
\begin{equation}\label{eq:Tukia-naturality}
 \mathcal E(\psi\circ g\circ\phi)
 =\psi\circ\mathcal E(g)\circ\phi.
\end{equation}

We next establish continuity at the identity in a form convenient for our arguments.

\begin{lemma}\label{lem:extension-continuity}
Let $\eta$ be a distortion function, and let $g_n\colon\R\to\C$ be normalized $\eta$-quasisymmetric embeddings.  
If $g_n\to\id_{\R}$ uniformly in the spherical metric, then the complex dilatations of $\mathcal{E}g_n$ satisfy
\[
 \mu_{\mathcal Eg_n}(z)\to 0 \qquad (n\to\infty)
\]
for every fixed $z\in\C\setminus\R$.
Moreover, this convergence is uniform on every compact subset of either open half-plane.
\end{lemma}

\begin{proof}
It suffices to prove the statement for the upper half-plane $\mathbb{H}$; the argument for the lower half-plane is completely analogous.
Write the extension as
\[
 \mathcal Eg_n=A_n\circ B_{h_n},
 \qquad h_n=A_n^{-1}\circ g_n,
\]
where $A_n=A_{g_n}$. 
Since the embeddings $g_n$ share the same distortion function $\eta$, the quasilines $\Gamma_n = g_n(\mathbb{R})$ have uniformly bounded turning.
Hence the normalized conformal maps $A_n\colon\mathbb H\to\Omega_{g_n}^+$ admit quasiconformal extensions with a uniform bound on their maximal dilatations. 
Consequently, the maps $A_n$ and $A_n^{-1}$ form normal families.

By normal-family compactness, every subsequence of $A_n$ contains a further subsequence that converges locally uniformly. 
Given that $g_n(\mathbb{R}) \to \mathbb{R}$ in the spherical metric, any subsequential limit of $A_n$ must map $\mathbb{H}$ conformally onto itself and fix $0, 1, \infty$. 
By uniqueness, this limit is therefore the identity. 
Consequently, 
\[
 A_n\to\id,
 \qquad A_n^{-1}\to\id,
 \qquad h_n\to\id,
\]
where the convergence of $A_n$ and $A_n^{-1}$ is local uniform in $\mathbb{H}$, and that of $h_n$ is uniform on compact intervals of $\mathbb{R}$.

Next, for $y>0$, differentiating the integral in \eqref{eq:alphabeta} gives
\[
 \partial_x\alpha_h=\frac{h(x+y)-h(x)}{y},\qquad\partial_y\alpha_h=\frac{h(x+y)-\alpha_h(x,y)}{y},
\]
and
\[
 \partial_x\beta_h=\frac{h(x)-h(x-y)}{y},\qquad\partial_y\beta_h=\frac{h(x-y)-\beta_h(x,y)}{y}.
\]
Thus the value and the first-order derivative of $B_h$ at a fixed point $z \in \mathbb H$
depend continuously on the restriction of $h$ to a fixed compact interval.  
It follows that
\[
 D(B_{h_n})(z)\to
 D(B_{\id})(z)=I \qquad (n\to\infty)
\]
uniformly on compact subsets of $\Hh$, where $D$ denotes the Jacobian operator and $I$ is the $2\times2$ identity matrix.  
Therefore
\[
 |\mu_{\mathcal Eg_n}(z)|
 =|\mu_{B_{h_n}}(z)|\to 0,
\]
as required.
\end{proof}

\subsection{Proof of Theorem~\ref{thm:main-intro}}

We now prove the quasiconformal extension theorem of $\mathrm{EWAS}$.

\begin{theorem}\label{thm:Esharp-AC}
Let $f\colon\R\to\C$ be a quasisymmetric $\mathrm{EWAS}$ embedding. 
Then its Tukia extension satisfies
\[
\esssup_{0<|\operatorname{Im}z|<r}
|\mu_{\mathcal Ef}(z)| \to 0 \qquad (r \to 0).
\]
\end{theorem}

\begin{proof}
Suppose, to the contrary, that the conclusion fails. 
After passing to one half-plane, say, to $\mathbb H$, there exist $\varepsilon_0>0$ and points
$z_n=x_n+iy_n\in\Hh$ with $y_n\to0$,
such that
\begin{equation}\label{eq:badmu}
 |\mu_{\mathcal Ef}(z_n)|\ge\varepsilon_0.
\end{equation}
Set
$\phi_n(t)=x_n-y_n+2y_nt$
and let
\[
 \psi_n(w)=
 \frac{w-f(x_n-y_n)}{f(x_n+y_n)-f(x_n-y_n)}
\]
be the similarity. 
Define
$g_n=\psi_n\circ f\circ\phi_n\colon\R\to\C$.
Then
$g_n(0)=0, g_n(1)=1, g_n(\infty)=\infty$, and all $g_n$ have the same distortion function $\eta$ as $f$.

By Lemma~\ref{lem:normal-boundary}, every subsequence contains a further
subsequence converging uniformly in the spherical metric to a normalized
quasisymmetric embedding $g$.  
Fix $u\in\R$ and $s>0$.  
The three points $\phi_n(u-s), \phi_n(u), \phi_n(u+s)$ have equal spacing $2y_ns$ and lie in an interval whose diameter tends to zero.  
The $\mathrm{EWAS}$ hypothesis for $f$ therefore gives
\[
 \frac{|g_n(u+s)-g_n(u)|}{|g_n(u-s)-g_n(u)|}\to 1 \qquad (n \to \infty).
\]
Passing to the limit yields
\[
 |g(u+s)-g(u)|=|g(u-s)-g(u)|
\]
for every $u$ and $s$.  
By Theorem~\ref{thm:MV}, $g$ is affine.  
The normalization forces $g=\id_{\R}$.  
Since every convergent subsequence has the same unique limit, we conclude that
$ g_n\to\id_{\R}$
uniformly in the spherical metric.

Let $w_0=(1+i)/2$.
Then $\phi_n(w_0)=z_n$. 
By the affine covariance \eqref{eq:Tukia-naturality},
\[
 \mathcal Eg_n
 =\psi_n\circ\mathcal Ef\circ\phi_n
\]
up to normalization.
Since $\phi_n$ and $\psi_n$ are conformally affine maps, 
it follows that $|\mu_{\mathcal Eg_n}(w_0)|=|\mu_{\mathcal Ef}(z_n)|$.
On the other hand, Combining $ g_n\to\id_{\R}$ with Lemma~\ref{lem:extension-continuity} implies that $\mu_{\mathcal Eg_n}(w_0)\to 0$, contradicting \eqref{eq:badmu}.
\end{proof}

\begin{remark}
In \cite{AndoniYang}, the asymptotic conformality of a bounded quasicircle $\Gamma$ is characterized by the asymptotic symmetry of the conformal homeomorphism of the unit disk $\mathbb D$ onto the interior domain of $\Gamma$ and its boundary extension.
\end{remark}

The following proposition shows that the asymptotic conformality yields asymptotic symmetry of the boundary map.

\begin{proposition}\label{prop:AC-boundary-AS}
Let $F\colon\C\to\C$ be a quasiconformal homeomorphism such that
\begin{equation}\label{eq:uniformAC}
 k_F(r)\coloneqq \esssup_{0<|\operatorname{Im}z|<r}|\mu_F(z)|\to 0\qquad(r\to0).
\end{equation}
Then the quasisymmetric embedding $f=F|_{\R}\colon\R\to\C$ is asymptotically symmetric.
\end{proposition}

\begin{proof}
Suppose to the contrary that $f$ is not asymptotically symmetric. Then there exist $\varepsilon_0>0$, $t>0$, and triples $a_n,b_n,x_n\in\R$ whose diameters tend to zero such that
\begin{equation}\label{eq:AS-failure-domain}
 |a_n-x_n|\leq t|b_n-x_n|
\end{equation}
and
\begin{equation}\label{eq:AS-failure-image}
 |f(a_n)-f(x_n)|>(1+\varepsilon_0)t|f(b_n)-f(x_n)|.
\end{equation}
Set $s_n=b_n-x_n\ne0$, and define the conformal affine maps
\[
\alpha_n(z)=x_n+s_nz, \qquad\beta_n(w)= \frac{w-f(x_n)}{f(b_n)-f(x_n)}.
\]
Note that $\alpha_n$ and $\beta_n$ remain conformal affine even when $s_n<0$.  
Consider the normalized sequence
\begin{equation}\label{eq:Gn}
 G_n=\beta_n\circ F\circ\alpha_n.
\end{equation}
Then $G_n(0)=0, G_n(1)=1,G_n(\infty)=\infty$, and all $G_n$ have the maximal dilatation as $F$.

Since $|\operatorname{Im}\alpha_n(z)|=|s_n|\,|\operatorname{Im}z|$ and $|s_n|\to0$, condition \eqref{eq:uniformAC} gives $\esssup_{z\in E}|\mu_{G_n}(z)|\to 0$ as $n \to \infty$ for every compact set $E\subset\C$.
By Lemma~\ref{lem:normal-qc}, every subsequential limit of $G_n$ is a conformal self-homeomorphism of $\C$ fixing $0,1,\infty$.  
Therefore,
$G_n\to \id_{\C}$
uniformly in the spherical metric.

Setting
\[
 u_n=\frac{a_n-x_n}{b_n-x_n}.
\]
we have $|u_n|\leq t$ by \eqref{eq:AS-failure-domain}. 
Passing to a subsequence if necessary, let $u_n\to u$ with $|u|\leq t$.  
It follows from $G_n\to \id_{\C}$ that $G_n(u_n)\to u$.

On the other hand, by \eqref{eq:AS-failure-image} and the definition of $G_n$ in \eqref{eq:Gn},
\[
 |G_n(u_n)|=\frac{|f(a_n)-f(x_n)|}{|f(b_n)-f(x_n)|}>(1+\varepsilon_0)t,
\]
which implies in the limit that $|u| \geq (1+\varepsilon_0)t > t$. This directly contradicts $|u|\leq t$, completing the proof.
\end{proof}

Combining the extension theorem and the boundary proposition gives the main
result.

\begin{proof}[Proof of Theorem~\ref{thm:main-intro}]
Only (iii)$\Rightarrow$(i) requires proof.  
If $f$ is $\mathrm{EWAS}$, then Theorem~\ref{thm:Esharp-AC} provides a quasiconformal extension satisfying
\eqref{eq:uniformAC}.  
Proposition~\ref{prop:AC-boundary-AS} implies that $f$ is $\mathrm{AS}$.
\end{proof}

\subsection{An alternative proof}
While the quasiconformal extension perspective is insightful, Theorem~\ref{thm:main-intro} also admits a short, direct proof. 

Suppose that $f$ is $\mathrm{EWAS}$ but not $\mathrm{AS}$.  
Then there exist $\varepsilon_0>0$, $t>0$, and triples $a_n,b_n,x_n\in\mathbb R$ such that
\[
|a_n-x_n|\leq t|b_n-x_n|,\qquad\operatorname{diam}\{a_n,b_n,x_n\}\to 0,
\]
but
\[
|f(a_n)-f(x_n)|
>(1+\varepsilon_0)t|f(b_n)-f(x_n)|.
\]
Set $s_n=b_n-x_n$ and define the normalized quasisymmetric embeddings
\[
g_n(u)=\frac{f(x_n+s_nu)-f(x_n)}{f(x_n+s_n)-f(x_n)}.
\]
By quasisymmetric compactness, after passing to a subsequence if necessary, $g_n$ converges locally uniformly on $\mathbb R$ to a normalized quasisymmetric embedding $g$.

We claim that $g$ satisfies condition \eqref{eq:isosceles} in Theorem \ref{thm:MV}.  
Indeed, for any $u,s\in\mathbb R$, the three points
\[
x_n+s_n(u+s),\qquad x_n+s_nu,\qquad x_n+s_n(u-s)
\]
form a symmetric triple, whose diameter is $2|s_n||s|$, which tends to $0$ as $n \to \infty$.  
The $\mathrm{EWAS}$ condition therefore gives
\[
\frac{|g_n(u+s)-g_n(u)|}{|g_n(u-s)-g_n(u)|}\to 1.
\]
Passing to the limit, we obtain
\[
|g(u+s)-g(u)|=|g(u-s)-g(u)|\qquad(u,s\in\mathbb R).
\]
By Theorem~\ref{thm:MV}, $g$ is affine.  
The normalization gives $g=\operatorname{id}_{\mathbb R}$.

Hence $g_n(u_n)\to u$, where $u_n=(a_n-x_n)/(b_n-x_n)\to u$ and $|u|\leq t$, contradicting the same image ratio as in the proof of Proposition~\ref{prop:AC-boundary-AS}.

This argument shows that the McKemie--Vaaler rigidity theorem alone
suffices for the boundary conclusion. The quasiconformal extension approach, on the other hand, provides additional geometric information, including the corresponding uniformly asymptotically conformal extension under the normalization.

\section{Proof of the other theorems and WAS--AS problem}\label{sec:proof_of_theorems_ref_thm_onto_line_intro}

In this section, we prove Theorems~\ref{thm:onto-line-intro}, \ref{thm:symmetric-source-intro}, and \ref{thm:symmetric-target-intro}.  
The arguments are given by generalizing the alternative proof of Theorem~\ref{thm:main-intro}
relying on Theorem \ref{thm:MV}.
We then distinguish between the noncompact and compact settings and explain the additional assumption that arises in the noncompact case. In the compact setting, these results turn out to be an affirmative answer to a special case of the problem by Brania and Yang.

\subsection{Proof of Theorem~\ref{thm:onto-line-intro}}

The proof uses quasiconformal compactness and the McKemie--Vaaler rigidity theorem. 
After normalization, $\mathrm{EWAS}$ passes to the limit and forces the limiting inverse to be affine.

\medskip

Assume that $h$ is not $\mathrm{AS}$. 
There are $\varepsilon_0>0$, $t>0$, and triples
$a_n,b_n,x_n\in\Gamma$ such that
\begin{equation}\label{eq:onto-failure-scale}
 |a_n-x_n|\leq t|b_n-x_n|, \qquad \diam\{a_n,b_n,x_n\}\to 0,
\end{equation}
and
\begin{equation}\label{eq:onto-failure-image}
 |h(a_n)-h(x_n)|>(1+\varepsilon_0)t|h(b_n)-h(x_n)|.
\end{equation}
Put $r_n=|b_n-x_n|\to0$ and define similarities
\[
 \Psi_n(z)=\frac{z-x_n}{b_n-x_n},
 \qquad
 \Phi_n(w)=\frac{w-h(x_n)}{h(b_n)-h(x_n)}.
\]
Choose a quasiconformal extension $H\colon\C\to\C$ of $h$, and set
\[
 H_n=\Phi_n\circ H\circ \Psi_n^{-1},\qquad\Gamma_n=\Psi_n(\Gamma).
\]
Then $H_n$ fixes $0,1,\infty$, maps $\Gamma_n$ onto $\R$, and has the same maximal dilatation as $H$ for every $n$. 
Passing to a subsequence, by Lemma~\ref{lem:normal-qc}, we may assume $H_n\to H_\infty$, and $H_n^{-1}\to H_\infty^{-1}$ 
uniformly in the spherical metric. 
Let $\Gamma_\infty=H_\infty^{-1}(\R)$, and put
\[
 p_n=H_n^{-1}|_{\R},\qquad  p_\infty=H_\infty^{-1}|_{\R}.
\]
Thus $p_n\to p_\infty$ locally uniformly on $\R$.

We claim that $p_\infty$ preserves every symmetric  triple.  
Fix $u\in\R$ and $s>0$, and write
\[
 d_n^+=|p_n(u+s)-p_n(u)|,
 \qquad
 d_n^-=|p_n(u-s)-p_n(u)|.
\]
Suppose, for contradiction, that along a subsequence
\[
 d_n^+\to  d^+,\qquad d_n^-\to  d^-,\qquad d^+>d^-.
\]
Then, for all sufficiently large $n$, $d_n^+>d_n^-$.  
The continuous function
\[
 v\longmapsto |p_n(u+v)-p_n(u)|,\qquad 0\leq v\leq s,
\]
takes the values $0$ and $d_n^+$ at the endpoints.  
Hence there is $v_n\in(0,s)$ such that
\begin{equation}\label{eq:inverse-equidistant}
 |p_n(u+v_n)-p_n(u)|=d_n^-=|p_n(u-s)-p_n(u)|.
\end{equation}
The three points in \eqref{eq:inverse-equidistant} remain in the fixed interval $[u-s,u+s]$.  
After applying $\Psi_n^{-1}$, their diameter is $O(r_n)$ and tends to $0$ as $n \to \infty$.  
The $\mathrm{EWAS}$ condition for $h$, applied to this symmetric  triple, gives
\[
\frac{v_n}{s}= \frac{|(u+v_n)-u|}{|(u-s)-u|}\to 1.
\]
Thus $v_n\to s$.  
Local uniform convergence of $p_n$ then implies
\[
 d^-=\lim_{n \to \infty}|p_n(u+v_n)-p_n(u)|=|p_\infty(u+s)-p_\infty(u)|=d^+,
\]
a contradiction.  
Swapping the roles of the two sides gives
\begin{equation*}\label{eq:pinfty-isosceles}
 |p_\infty(u+s)-p_\infty(u)|
 =|p_\infty(u-s)-p_\infty(u)|
\end{equation*}
for every $u$ and $s$.

By Theorem~\ref{thm:MV}, $p_\infty$ is affine.  
Since it fixes $0$ and $1$, we have $p_\infty=\id_{\R}$.  
Consequently,
\[
 \Gamma_\infty=\R,\qquad H_\infty|_{\R}=\id_{\R}.
\]
Finally, set
\[
 u_n=\Psi_n(a_n)=\frac{a_n-x_n}{b_n-x_n}.
\]
By \eqref{eq:onto-failure-scale}, $|u_n|\leq t$.  
Passing to a further subsequence, $u_n\to u\in\R$ with $|u|\leq t$.  
Therefore,
\[
 H_n(u_n)\to  H_\infty(u)=u.
\]
On the other hand, \eqref{eq:onto-failure-image} says
\[
 |H_n(u_n)|=\frac{|h(a_n)-h(x_n)|}{|h(b_n)-h(x_n)|}>(1+\varepsilon_0)t,
\]
which is impossible.  
Thus $h$ is $\mathrm{AS}$.

\subsection{Proof of Theorem~\ref{thm:symmetric-source-intro}}

Theorem~\ref{thm:symmetric-source-intro} extends Theorem~\ref{thm:main-intro} from the real line $\mathbb{R}$ to arbitrary asymptotically conformal quasilines.  
The argument uses the asymptotic conformality of the quasiline $\Gamma_1$ to identify the limiting geometry.

\medskip

Suppose that $\mathrm{AS}$ fails.  
Choose $\varepsilon_0>0$, $t>0$, and triples $a_n,b_n,x_n\in\Gamma_1$ such that
\begin{equation*}
 |a_n-x_n|\leq t|b_n-x_n|,\qquad \diam\{a_n,b_n,x_n\}\to0,
\end{equation*}
and
\begin{equation}\label{eq:sym-source-failure-image}
 |h(a_n)-h(x_n)|>(1+\varepsilon_0)t|h(b_n)-h(x_n)|.
\end{equation}
Define similarities as before
\[
 \Psi_n(z)=\frac{z-x_n}{b_n-x_n},
 \qquad
 \Phi_n(w)=\frac{w-h(x_n)}{h(b_n)-h(x_n)}.
\]
Fix a quasiconformal extension $H$ of $h$ and set
\[
 H_n=\Phi_n\circ H\circ \Psi_n^{-1},
 \quad
 \Gamma_{1,n}=\Psi_n(\Gamma_1),
 \quad
 \Gamma_{2,n}=\Phi_n(\Gamma_2).
\]
After passing to a subsequence, $H_n\to H_\infty$ uniformly in the spherical metric.  
Let
\[
 \Gamma_{j,\infty}=\lim_{n \to \infty}\Gamma_{j,n} \quad (j=1,2),\qquad  h_\infty=H_\infty|_{\Gamma_{1,\infty}}.
\]
By Lemma~\ref{lem:symmetric-blowup}, $\Gamma_{1,\infty}=\R$.
The normalized quasiconformal parametrizations used in taking the limit also preserve the order of the two components of $\Gamma_{1,n}\setminus\{c_n\}$; hence the ordered-convergence hypothesis in
Lemma~\ref{lem:equidistant-approx} is available below.

We show that $h_\infty\colon\R\to\C$ preserves all symmetric  triples.  
Fix $x\in\R$ and $s>0$.  
Choose $c_n\in\Gamma_{1,n}$ with $c_n\to x$ and use Lemma~\ref{lem:equidistant-approx} to find $a_n',b_n'\in\Gamma_{1,n}$ such that
\[
 a_n'\to x+s,\qquad b_n'\to x-s,\qquad |a_n'-c_n|=|b_n'-c_n|.
\]
These normalized points stay bounded. Moreover, their preimages under $\Psi_n^{-1}$ have diameter $O(|b_n-x_n|)\to0$
as $n \to \infty$.  
Applying $\mathrm{EWAS}$ to the symmetric  triple and passing to the limit yields
\[
 |h_\infty(x+s)-h_\infty(x)|=|h_\infty(x-s)-h_\infty(x)|.
\]
Theorem~\ref{thm:MV} now implies that $h_\infty|_{\mathbb{R}}$ is affine. 
Since $H_n(0)=0$, $H_n(1)=1$, and $H_n(\infty)=\infty$, the normalization gives $h_\infty|_{\mathbb{R}}=\id_{\R}$.
In particular, $\Gamma_{2,\infty}=\R$.

Set $u_n=\Psi_n(a_n)$.
Then $|u_n|\leq t$.  
Passing to a subsequence, $u_n\to u\in\R$ with $|u|\leq t$.  
Hence
\[
 H_n(u_n)\to  H_\infty(u)=u.
\]
But \eqref{eq:sym-source-failure-image} gives $|H_n(u_n)|>(1+\varepsilon_0)t$,
a contradiction.  
Therefore $h$ is $\mathrm{AS}$.

\begin{remark}
Theorem~\ref{thm:symmetric-source-intro} gives a noncompact positive result for the $\mathrm{WAS}$--$\mathrm{AS}$ problem under asymptotic conformality of the quasiline.
Theorem~\ref{thm:onto-line-intro} is complementary: no asymptotic conformality assumption on the quasiline is needed when the image is $\mathbb{R}$.
\end{remark}

\subsection{Proof of Theorem~\ref{thm:symmetric-target-intro}}

Theorem~\ref{thm:symmetric-source-intro} assumes asymptotic conformality of $\Gamma_1$.  
It is natural to ask whether the formally dual statement, with asymptotic conformality imposed only on $\Gamma_2$, is also true.  
The noncompact setting makes this direction more delicate.

To explain the point, suppose that $h\colon \Gamma_1 \to \Gamma_2$ is a quasisymmetric homeomorphism between quasilines and that $\mathrm{AS}$ fails.
As in the proof of Theorem~\ref{thm:symmetric-source-intro}, we may choose $\varepsilon_0>0$, $t>0$, and triples $a_n,b_n,x_n\in\Gamma_1$ such that
\[
 |a_n-x_n|\leq t|b_n-x_n|,\qquad \diam\{a_n,b_n,x_n\}\to 0,
\]
whereas
\[
 |h(a_n)-h(x_n)| >(1+\varepsilon_0)t\,|h(b_n)-h(x_n)|.
\]

The normalization is made at the scale $r_n=|b_n-x_n|\to0$.  
When $\Gamma_1$ is asymptotically conformal, the normalized curves converge to a straight line, which is the key step in the proof of Theorem~\ref{thm:symmetric-source-intro}.

If instead only $\Gamma_2$ is assumed to be asymptotically conformal, the
corresponding scale is
\[
 \rho_n=|h(b_n)-h(x_n)|.
\]
To apply the infinitesimal flatness of $\Gamma_2$, one needs $\rho_n\to 0$.

For a quasisymmetric homeomorphism  between noncompact quasilines, however, the implication
\[
 |b_n-x_n|\to 0\quad\implies\quad|h(b_n)-h(x_n)|\to 0
\]
is not automatic.  
In particular, global quasisymmetry does not imply uniform continuity in the noncompact setting.  
Thus the argument with asymptotic conformality of $\Gamma_2$ does not follow immediately from that of Theorem~\ref{thm:symmetric-source-intro}.

The real line is exceptional in this setting.  
In Theorem~\ref{thm:onto-line-intro}, no smallness assumption on the image scale is needed, since every affine normalization of $\R$ is again $\R$.  
For a general asymptotically conformal quasiline, however, the corresponding normalized curves approach a straight line only when the normalization scale tends to zero.  
Thus, in this case, one needs $\rho_n\to0$.

The obstruction disappears under the additional assumption of uniform continuity.

\begin{proof}[Proof of Theorem~\ref{thm:symmetric-target-intro}]
We briefly indicate the proof, which combines the arguments from Theorems~\ref{thm:onto-line-intro} and \ref{thm:symmetric-source-intro}.

Suppose that $\mathrm{AS}$ fails.  Choose $\varepsilon_0>0$, $t>0$, and triples $a_n,b_n,x_n\in\Gamma_1$ such that
\begin{equation}\label{eq:sym-target-failure}
 |a_n-x_n|\leq t|b_n-x_n|,\qquad \diam\{a_n,b_n,x_n\}\to0,
\end{equation}
and
\begin{equation}\label{eq:sym-target-image-failure}
 |h(a_n)-h(x_n)|>(1+\varepsilon_0)t\,|h(b_n)-h(x_n)|.
\end{equation}
Set $r_n=|b_n-x_n|$ and $\rho_n=|h(b_n)-h(x_n)|$.
Then $r_n\to0$, and uniform continuity of $h$ gives $\rho_n\to 0$.

Define the similarities $\Psi_n(z)$ and $\Phi_n(w)$ as in the proof of
Theorem~\ref{thm:symmetric-source-intro}.

Fix a quasiconformal extension $H\colon\C\to\C $ of $h$ with $H(\infty)=\infty$ and put $H_n=\Phi_n\circ H\circ \Psi_n^{-1}$.
After passing to a subsequence, we may assume that $H_n\to H_\infty$ and $H_n^{-1}\to H_\infty^{-1}$ uniformly in the spherical metric. 
The normalized curves $\Gamma_{1,n}=\Psi_n(\Gamma_1)$ and $\Gamma_{2,n}=\Phi_n(\Gamma_2)$ converge to quasilines $\Gamma_{1,\infty}$ and $\Gamma_{2,\infty}$, respectively.

By Lemma~\ref{lem:symmetric-blowup}, applied to $\Gamma_2$ with the scale $\rho_n\to0$, we obtain $\Gamma_{2,\infty}=\R$.
Consequently, the limiting boundary map $h_\infty$ maps $\Gamma_{1,\infty}$ onto $\mathbb R$.

We claim that its inverse
\[
 p_\infty=h_\infty^{-1}\colon\R\to\Gamma_{1,\infty}
\]
preserves all symmetric  triples.  
This requires a slight modification of the argument in the proof of Theorem~\ref{thm:onto-line-intro}, because the curves $\Gamma_{2,n}$ are not equal to $\R$, but only converge to $\R$.

Put
\[
 p_n=H_n^{-1}|_{\Gamma_{2,n}}\colon\Gamma_{2,n}\to\Gamma_{1,n}.
\]
Fix $u\in\R$ and $s>0$.  
Using the normalized parametrizations of $\Gamma_{2,n}$ and preserving the order of the two components, choose points $c_n,q_n^+,q_n^-\in\Gamma_{2,n}$ such that
\[
 c_n\to u,\qquad q_n^+\to u+s,\qquad q_n^-\to u-s,
\]
where $q_n^+$ and $q_n^-$ lie on the two different components of $\Gamma_{2,n}\setminus\{c_n\}$.  
By the convergence of $H_n^{-1}$,
\[
 d_n^\pm\coloneqq  |p_n(q_n^\pm)-p_n(c_n)|\to  d^\pm \coloneqq  |p_\infty(u\pm s)-p_\infty(u)|.
\]

Suppose, for contradiction, that $d^+>d^-$.  
Then $d_n^+>d_n^-$ for all sufficiently large $n$.  
Along the compact subarc of $\Gamma_{2,n}$ joining $c_n$ to $q_n^+$, the function
\[
 q\longmapsto |p_n(q)-p_n(c_n)|
\]
is continuous, takes the value $0$ at $c_n$, and the value $d_n^+$ at $q_n^+$.  
Hence there exists a point $q_n'$ on this subarc such that
\begin{equation}\label{eq:target-intermediate}
 |p_n(q_n')-p_n(c_n)|=|p_n(q_n^-)-p_n(c_n)|.
\end{equation}
The points involved stay in bounded parts of the normalized curves.
After applying $A_n^{-1}$, the three points in \eqref{eq:target-intermediate} therefore have diameter $O(r_n)\to0$
as $n \to \infty$.  
Since they form a symmetric  triple in the domain,
the $\mathrm{EWAS}$ condition for $h$ gives
\begin{equation*}
 \frac{|q_n'-c_n|}{|q_n^--c_n|}\to 1.
\end{equation*}

Since $c_n\to u$ and $q_n^-\to u-s$, we have $|q_n^--c_n|\to s$, and hence
$|q_n'-c_n|\to s$.
Moreover, $q_n'$ lies on the component of $\Gamma_{2,n}\setminus\{c_n\}$ converging to the positive half-line.
The ordered convergence $\Gamma_{2,n}\to\R$ therefore implies $q_n'\to  u+s$.
Using the convergence of $H_n^{-1}$ again and \eqref{eq:target-intermediate}, we obtain
\[
 d^+=|p_\infty(u+s)-p_\infty(u)|=\lim_{n\to\infty}|p_n(q_n')-p_n(c_n)|=d^-,
\]
a contradiction.  
Interchanging the two sides gives
\[
 |p_\infty(u+s)-p_\infty(u)|=|p_\infty(u-s)-p_\infty(u)|\qquad (u\in\R,\ s>0).
\]

By Theorem \ref{thm:MV}, $p_\infty$ is affine. 
The normalization at $0$, $1$, and $\infty$ then yields $\Gamma_{1,\infty}=\R$ and $h_\infty=\id_{\R}$.

Finally, set $u_n=\Psi_n(a_n)$.
By \eqref{eq:sym-target-failure}, we have $|u_n|\leq t$.
Passing to a further subsequence, $u_n\to u\in\R$ with $|u|\leq t$.
Hence $H_n(u_n)\to H_\infty(u)=u$.
On the other hand, \eqref{eq:sym-target-image-failure} gives
\[
 |H_n(u_n)|=\frac{|h(a_n)-h(x_n)|}{|h(b_n)-h(x_n)|}>(1+\varepsilon_0)t,
\]
which is  a contradiction.  
Thus $h$ is $\mathrm{AS}$.
\end{proof}

\subsection{The compact setting}

The situation becomes symmetric again in the compact case, since uniform continuity is then automatic.
Although an argument by passing to the universal covering spaces may derive claims for quasicircle cases from quasiline cases, below in the proof, we record how this follows from our previous arguments. 
We only need the following minor modifications.
For the normalization of mappings on $\mathbb S$, we assume that three points $1$, $i$, and $-i$ are fixed. 
To apply Theorem \ref{thm:MV}, we use condition \eqref{eq:isosceles-circle} and assert that a weakly $1$-quasisymmetric embedding of $\mathbb S$ is circular.

\begin{theorem}\label{thm:compact-one-symmetric}
Let $\Gamma_1$ and $\Gamma_2$ be bounded quasicircles in $\C$, and let $h\colon\Gamma_1\to\Gamma_2$ be a quasisymmetric homeomorphism. 
Assume that at least one of $\Gamma_1$ and $\Gamma_2$ is an asymptotically conformal quasicircle.  
If $h$ is $\mathrm{EWAS}$, then $h$ is $\mathrm{AS}$.  
Consequently, the conditions $\mathrm{EWAS}$, $\mathrm{WAS}$, and $\mathrm{AS}$ are equivalent in this case.
\end{theorem}

\begin{proof}
Suppose first that $\Gamma_1$ is asymptotically conformal.  
The argument in the proof of Theorem~\ref{thm:symmetric-source-intro} applies with a little modification mentioned above.
Indeed, after normalizing a sequence of symmetric triples in $\Gamma_1$ whose diameters tend to zero, the corresponding normalized curves converge to a circle. This is due to 
the circle analogue of Lemma \ref{lem:symmetric-blowup}, which follows by the same argument, using the smaller-diameter subarc. 
The limiting map is therefore a weakly $1$-quasisymmetric embedding of $\mathbb S$ into $\C$, and Theorem \ref{thm:MV} forces it to be circular.  
The normalized three-point contradiction then proves $\mathrm{AS}$.

Suppose next that $\Gamma_2$ is asymptotically conformal.  
Since $\Gamma_1$ is compact, the homeomorphism $h\colon\Gamma_1\to\Gamma_2$ is uniformly continuous.  
Hence, for every sequence $b_n,x_n\in\Gamma_1$ with $|b_n-x_n|\to 0$, we have $|h(b_n)-h(x_n)|\to 0$.
Thus the normalization scale $\rho_n=|h(b_n)-h(x_n)|$ tends to zero.  
The proof of Theorem~\ref{thm:symmetric-target-intro} therefore applies. 
The normalized curves converge to a circle, and the limiting map $h_\infty\colon\Gamma_{1,\infty}\to\mathbb S$ is treated by the inverse-map argument from the proof of Theorem~\ref{thm:onto-line-intro}.  
Theorem~\ref{thm:MV} then identifies the normalized limit with the identity, contradicting the assumed failure of $\mathrm{AS}$.

Thus $\mathrm{EWAS}$ implies $\mathrm{AS}$ whenever either $\Gamma_1$ or $\Gamma_2$ is asymptotically conformal. 
Since $\mathrm{AS}\implies\mathrm{WAS}\implies\mathrm{EWAS}$, the final assertion follows.
\end{proof}

Brania and Yang \cite[p.2678]{BraniaYang} raised the problem of whether a weakly asymptotically symmetric embedding of a bounded Jordan curve into the plane is asymptotically symmetric, and also a particular case that they consider tractable.
Theorem \ref{thm:compact-one-symmetric} gives an affirmative answer to the case explicitly suggested by them, namely quasisymmetric homeomorphisms between asymptotically conformal quasicircles.

\section{A proof by modulus of curve families}\label{sec:a_proof_by_modulus_of_curve_families}

We now explain how the modulus argument in \cite[Theorem~3.1]{BraniaYang} can be adapted to the present setting.  
This section is not needed for the proof above, but it clarifies the connection with the original 
argument for $\mathrm{WAS}$--$\mathrm{AS}$ problem.

For a curve family $\Gamma$ in $\C$, we consider its modulus $\Mod \Gamma$.
For disjoint continua $E$ and $F$ in $\C$, let $\Delta(E,F;\C)$ denote the curve family joining $E$ and $F$ in $\C$.
Let $\lambda$ denote the Teichm\"uller function: if $E=[-1,0]$ and $F=[s,\infty]$, $s>0$, then $\lambda(s)=\Mod\Delta(E,F;\C)$.  
The function $\lambda$ is continuous and strictly decreasing, with $\lambda(0)=\infty$ and $\lambda(\infty)=0$.
We also use the following comparison principle: if $a,b\in E$, and $c,d\in F$, then
\begin{equation}\label{eq:Teich-comparison}
 \Mod\Delta(E,F;\C)\ge\lambda([a,b,c,d]),
\end{equation}
where $[a,b,c,d]$ is the absolute value of the cross ratio:
\[
 [a,b,c,d] =\frac{|b-c|\,|a-d|}{|a-b|\,|c-d|}.
\]
These are precisely the modulus facts recalled before the proof of \cite[Theorem~3.1]{BraniaYang}.

Let $\gamma\colon X\to\Gamma$ be a parametrization of a quasiline by a metric space $(X,d_X)$. 
We say that $\Gamma$ is \emph{asymptotically conformal relative to
$\gamma$} if
\[
\lim_{t\to 0}\,\sup_{z_1,z_2\in \Gamma}\,\left\{\max_{z\in\Gamma[z_1,z_2]}\frac{|z_1-z|+|z-z_2|}{|z_1-z_2|}:0<d_X(\gamma^{-1}(z_1),\gamma^{-1}(z_2))\leq t\right\}=1.
\]
Here $\Gamma[z_1,z_2]$ denotes the subarc of $\Gamma$ joining $z_1$ and $z_2$. 
When $\gamma$ is the identity parametrization of $\Gamma$, this reduces to the usual notion of asymptotic symmetry. 
Concerning this condition, see \cite{MT1}.

We first note the following geometric consequence of asymptotic conformality of quasiconformal mappings.

\begin{lemma}\label{lem:asymptotic-symmetry}
Under the hypotheses of Proposition~\ref{prop:AC-boundary-AS}, the curve $\Gamma=F(\R)$ is an asymptotically conformal quasiline relative to $f=F|_{\R}$, that is, for any $p_n<q_n$ and $r_n\in[p_n,q_n]$ in $\R$,
\[
 \frac{|f(p_n)-f(r_n)|+|f(r_n)-f(q_n)|}{|f(p_n)-f(q_n)|} \to 1 \quad {\rm as} \quad |p_n-q_n|\to 0
\]
uniformly independent of the choice of the sequences.
\end{lemma}

\begin{proof}
Normalize $p_n$ and $q_n$ to $0$ and $1$ exactly as in \eqref{eq:Gn} to define the normalized quasiconformal maps $G_n$ extending the quasisymmetric embeddings. 
By the same argument as in the proof of Proposition~\ref{prop:AC-boundary-AS}, $G_n\to \id_{\C}$ 
locally uniformly. 
Writing
\[
 \sigma_n=\frac{r_n-p_n}{q_n-p_n}\in[0,1]
\]
and passing to a subsequence with $\sigma_n\to\sigma$, the displayed ratio converges to $|\sigma|+|1-\sigma|=1$.
A contradiction argument gives uniformity.
\end{proof}
\begin{proof}[Another proof of Proposition \ref{prop:AC-boundary-AS}]
Assume that $\mathrm{AS}$ fails and choose sequences as in \eqref{eq:AS-failure-domain}--\eqref{eq:AS-failure-image}.  
Passing to a subsequence, suppose
\begin{equation}\label{eq:t0-t0prime}
 \frac{|a_n-x_n|}{|b_n-x_n|}\to  t_0,
 \qquad
 \frac{|f(a_n)-f(x_n)|}{|f(b_n)-f(x_n)|}\to  t_0',
\end{equation}
where $0\le t_0<t_0'<\infty$.
The finiteness of $t_0'$ follows from global quasisymmetry. 
There are three possible linear-order configurations, corresponding to the configurations AXB, XAB, and ABX in \cite{BraniaYang}.
We will consider only the configuration AXB. The reduction of the other cases to AXB is
the same as in the original argument by using Lemma \ref{lem:asymptotic-symmetry}.

Suppose $x_n$ lies between $a_n$ and $b_n$.  
After reversing orientation if necessary, assume $a_n<x_n<b_n$. 
Fix a small number $\rho>0$.  
For all sufficiently large $n$, put $d_n=x_n+2\rho$, and let
\[
 E_n=[a_n,x_n],\qquad F_n=[b_n,d_n],\qquad\Gamma_n=\Delta(E_n,F_n;\C).
\]
Since the four points are collinear and ordered,
\begin{equation*}
 \Mod\Gamma_n=\lambda(\tau_n),\qquad\tau_n=\frac{|x_n-b_n|\,|a_n-d_n|}{|a_n-x_n|\,|b_n-d_n|}.
\end{equation*}
Since $|a_n-x_n|,|b_n-x_n|\to0$ while $\rho>0$ is fixed, it follows that
\begin{equation}\label{eq:tau-limit}
 \tau_n\to \frac1{t_0}.
\end{equation}

Let $E_n'=f(E_n)$, $F_n'=f(F_n)$, and $\Gamma_n'=F(\Gamma_n)$.
The comparison principle \eqref{eq:Teich-comparison} gives
\begin{equation}\label{eq:image-lower}
 \Mod\Gamma_n'\ge\lambda(\tau_n'),
\end{equation}
where
\[
 \tau_n'=\frac{|f(x_n)-f(b_n)|\,|f(a_n)-f(d_n)|}{|f(a_n)-f(x_n)|\,|f(b_n)-f(d_n)|}.
\]
Global quasisymmetry gives
\[
 \frac{|f(a_n)-f(x_n)|}{|f(d_n)-f(x_n)|}\leq\eta\!\left(\frac{|a_n-x_n|}{|d_n-x_n|}\right)=\eta\!\left(\frac{|a_n-x_n|}{2\rho}\right)\to 0,
\]
and similarly
\[
 \frac{|f(b_n)-f(x_n)|}{|f(d_n)-f(x_n)|}\leq\eta\!\left(\frac{|b_n-x_n|}{|d_n-x_n|}\right)=\eta\!\left(\frac{|b_n-x_n|}{2\rho}\right)\to 0.
\]

By the triangle inequality,
\[
1-\frac{|f(a_n)-f(x_n)|}{|f(d_n)-f(x_n)|}\leq\frac{|f(a_n)-f(d_n)|}{|f(d_n)-f(x_n)|}\leq1+\frac{|f(a_n)-f(x_n)|}{|f(d_n)-f(x_n)|},
\]
and hence
\[
\frac{|f(a_n)-f(d_n)|}{|f(d_n)-f(x_n)|}\to 1.
\]
The same argument with $a_n$ replaced by $b_n$ gives
\[
\frac{|f(b_n)-f(d_n)|}{|f(d_n)-f(x_n)|}\to 1,
\]
and therefore
\[
\frac{|f(a_n)-f(d_n)|}{|f(b_n)-f(d_n)|}\to 1.
\]
Since
\[
\tau_n'=\frac{|f(b_n)-f(x_n)|}{|f(a_n)-f(x_n)|}\frac{|f(a_n)-f(d_n)|}{|f(b_n)-f(d_n)|},
\]
the first factor tends to $1/t_0'$ by \eqref{eq:t0-t0prime} while the second tends to $1$. 
Thus
\begin{equation}\label{eq:tauprime-limit}
\tau_n'\to \frac1{t_0'}.
\end{equation}

We next estimate $\Mod\Gamma_n'$ from above.  
Split $\Gamma_n=\Gamma_{n,1}\cup\Gamma_{n,2}$, where $\Gamma_{n,1}$ consists of curves contained in $B(x_n,\rho)$ and $\Gamma_{n,2}$ consists of the remaining curves.  
Set
\[
 K(\rho)=\frac{1+k_F(\rho)}{1-k_F(\rho)}.
\]
Since $B(x_n,\rho)$ lies in the strip $|\operatorname{Im}z|<\rho$, quasi-invariance of modulus gives
\begin{equation}\label{eq:local-modulus}
 \Mod F(\Gamma_{n,1})\leq K(\rho)\Mod\Gamma_{n,1}\leq K(\rho)\lambda(\tau_n).
\end{equation}

Let $r_n=2\max\{|a_n-x_n|,|b_n-x_n|\}$.
Every curve in $\Gamma_{n,2}$ contains a subcurve joining the circles $|z-x_n|=r_n$ and $|z-x_n|=\rho$.  
Hence the annulus estimate gives
\begin{equation}\label{eq:escaping-modulus}
 \Mod\Gamma_{n,2}\leq\frac{2\pi}{\log(\rho/r_n)}\to 0.
\end{equation}
If $K_0$ is the global maximal dilatation of $F$, then
\begin{equation}\label{eq:escaping-image}
 \Mod F(\Gamma_{n,2})\leq K_0\Mod\Gamma_{n,2}\to 0.
\end{equation}
By subadditivity of modulus, \eqref{eq:image-lower}, \eqref{eq:local-modulus}, and \eqref{eq:escaping-image},
\[
 \lambda(\tau_n')\leq K(\rho)\lambda(\tau_n)+o(1).
\]
Letting $n\to\infty$ and using \eqref{eq:tau-limit} and \eqref{eq:tauprime-limit}, we obtain $\lambda(1/t_0')\leq K(\rho)\lambda(1/t_0)$.

Finally let $\rho\to0$.  Since $K(\rho)\to1$, $\lambda(1/t_0')\le\lambda(1/t_0)$, which is impossible because $t_0<t_0'$ and $\lambda$ is strictly decreasing.
The case $t_0=0$ is interpreted using $\lambda(\infty)=0$ and is even more immediate.
\end{proof}

We have therefore obtained a modulus-theoretic proof of Proposition~\ref{prop:AC-boundary-AS}.  
The two estimates \eqref{eq:local-modulus} and \eqref{eq:escaping-modulus} are the real-line counterparts of the two central estimates in Brania and Yang: the first uses almost conformality near the boundary, while the second makes the family of curves escaping the controlled neighborhood have arbitrarily small modulus.

\end{document}